\documentclass[11pt]{article}
\usepackage[vcentermath,enableskew]{youngtab}
\usepackage{amssymb}
\usepackage{amsthm}
\usepackage{amsmath}

\newtheorem{theorem}{Theorem}[section]
\newtheorem{lemma}{Lemma}[section]

\numberwithin{equation}{section}

\newcommand{\cP}{\mathcal{P}}
\newcommand{\cU}{\mathcal{U}}
\newcommand{\tF}{\tilde{F}}
\newcommand{\FF}{\mathbb{F}}  
\newcommand{\GL}{\mathrm{GL}}
\newcommand{\U}{\mathrm{U}}
\newcommand{\SL}{\mathrm{SL}}
\newcommand{\PGL}{\mathrm{PGL}}
\newcommand{\SU}{\mathrm{SU}}
\newcommand{\PGU}{\mathrm{PGU}}

\def\adots{\mathinner{\mkern2mu\raise0pt\hbox{.}  
\mkern2mu\raise4pt\hbox{.}\mkern1mu
\raise7pt\vbox{\kern7pt\hbox{.}}\mkern1mu}}

\begin{document}

\bibliographystyle{amsplain}

\title{On the number of real classes in the finite projective linear and unitary groups}
\author{Elena Amparo and C. Ryan Vinroot}
\date{}

\maketitle

\begin{abstract}
We show that for any $n$ and $q$, the number of real conjugacy classes in $\PGL(n, \FF_q)$ is equal to the number of real conjugacy classes of $\GL(n, \FF_q)$ which are contained in $\SL(n, \FF_q)$, refining a result of Lehrer, and extending the result of Gill and Singh that this holds when $n$ is odd or $q$ is even.  Further, we show that this quantity is equal to the number of real conjugacy classes in $\PGU(n, \FF_q)$, and equal to the number of real conjugacy classes of $\U(n, \FF_q)$ which are contained in $\SU(n,\FF_q)$, refining results of Gow and Macdonald.  We also give a generating function for this common quantity.
\\
\\
\noindent 2010 {\it Mathematics Subject Classification: } 20G40, 20E45, 05A15
\end{abstract}

\section{Introduction}

It was proved by G. I. Lehrer \cite{Le75} that the number of conjugacy classes in the finite projective linear group $\PGL(n, \FF_q)$ is equal to the number of conjugacy classes of the finite general linear group $\GL(n, \FF_q)$ which are contained in the finite special linear group $\SL(n, \FF_q)$.  I. G. Macdonald \cite{Ma81} showed that the number of conjugacy classes in the finite projective unitary group $\PGU(n, \FF_q)$ is equal to the number of conjugacy classes of the finite unitary group $\U(n, \FF_q)$ which are contained in the finite special unitary group $\SU(n, \FF_q)$ (although this number is different than the number of conjugacy classes in $\PGL(n, \FF_q)$ in general).  

Meanwhile, R. Gow \cite{Go81} considered the number of \emph{real} conjugacy classes in $\GL(n, \FF_q)$ and in $\U(n, \FF_q)$, where a conjugacy class of a finite group $G$ is real if whenever $g$ is in the class, then so is $g^{-1}$.  In particular, Gow noted \cite[pg. 181]{Go81} that the number of real classes of $\GL(n, \FF_q)$ is equal to the number of real classes of $\U(n, \FF_q)$.

More recently, N. Gill and A. Singh \cite{GiSi111, GiSi112} classified the real conjugacy classes of $\PGL(n, \FF_q)$ and $\SL(n, \FF_q)$.  They noted \cite[after Theorem 2.8]{GiSi112} that when $q$ is even or $n$ is odd, the number of real classes of $\PGL(n, \FF_q)$ is equal to the number of real classes of $\GL(n, \FF_q)$ which are contained in $\SL(n, \FF_q)$.  In this paper we prove that this equality holds for all $n$ and $q$.  Moreover, we show that this number is equal to the number of real classes of $\PGU(n, \FF_q)$, and is equal to the number of real classes of $\U(n, \FF_q)$ which are contained in $\SU(n, \FF_q)$.  That is, we extend the result of Gill and Singh, and give a refinement of the results of Lehrer, Macdonald, and Gow.  

This paper is organized as follows.  In Section \ref{Prelims}, we establish notation for partitions and linear and unitary groups over finite fields, and we give an overview of the sets of polynomials over finite fields which we need.  We finish this section with an important enumeration in Lemma \ref{ZetaUSelfCount}.  In Section \ref{RealClasses}, we describe the real conjugacy classes for the groups of interest.  In particular, in Section \ref{RealClassesSLSU} we give known results for the real classes in $\GL(n, \FF_q)$ and $\U(n, \FF_q)$, Gill and Singh's enumeration of the number of real classes of $\GL(n, \FF_q)$ which are contained in $\SL(n, \FF_q)$, and we explain why this is equal to the number of real classes of $\U(n, \FF_q)$ which are contained in $\SU(n ,\FF_q)$.  In Section \ref{RealClassesPGLPGU}, we classify the real classes of $\PGU(n, \FF_q)$ by following the methods of Gill and Singh for $\PGL(n, \FF_q)$.  In Lemma \ref{PGUPGLReal} we give an enumeration of the real classes of $\PGU(n, \FF_q)$, and show this is equal to the number of real classes of $\PGL(n, \FF_q)$.  Finally, in Section \ref{MainResult} we prove our main result in Theorem \ref{main}, where the work left to be done is to prove that the number of real classes of $\PGL(n, \FF_q)$ is equal to the number of real classes of $\GL(n, \FF_q)$ which are contained in $\SL(n, \FF_q)$.  We accomplish this by computing a generating function for each quantity, which has a particularly nice form.
\\
\\
\noindent{\bf Acknowledgements. }  The second-named author was supported in part by a grant from the Simons Foundation, Award \#280496.

\section{Preliminaries} \label{Prelims}

For positive integers $n, m$, we denote their greatest common divisor by $(n, m)$.  We let $|n|_2$ denote the largest power of $2$ which divides $n$, or the $2$-part of $n$.  If $G$ is a group with $g \in G$, then $|g|$ will denote the order of the element $g$, and $|g|_2$ will denote the $2$-part of the order of $g$.

\subsection{Partitions}  Given an integer $n \geq 0$, we denote a \emph{partition} $\nu$ of $n$ as 
$$ \nu = (1^{m_1} 2^{m_2} 3^{m_3} \cdots), $$
such that $\sum_{i \geq 1} im_i = n$.  Each integer $m_i = m_i(\nu) \geq 0$ is the \emph{multiplicity} of the part $i$ in $\nu$.  We can also denote the partition $\nu$ by
$$ \nu = (\nu_1, \nu_2, \ldots, \nu_l),$$
such that $\sum_{j=1}^l \nu_j = n$, and $\nu_j \geq \nu_{j+1} \geq 0$ for $j < l$.  Then we have $m_i(\nu)$ is the number of $j$ such that $\nu_j = i$.  We also assume each $\nu_j > 0$ unless $n=0$, in which the unique partition of $0$ is considered the empty partition.  We let $\cP_n$ denote the collection of all partitions of $n$.

\subsection{Linear and unitary groups over finite fields}  For any prime power $q$, we let $\FF_q$ denote a finite field with $q$ elements, and we fix an algebraic closure $\bar{\FF}_q$.  We let $\FF_q^{\times}$ and $\bar{\FF}_q^{\times}$ denote the multiplicative groups of nonzero elements in these fields.

Let $\GL(n, \bar{\FF}_q)$ denote the group of invertible $n$-by-$n$ matrices over $\bar{\FF}_q$, and we identify $\GL(1, \bar{\FF}_q )$ with $\bar{\FF}_q^{\times}$.  Define the standard Frobenius map $F$ on $\GL(n, \bar{\FF}_q)$ by $F((a_{ij}))=(a_{ij}^q)$, and so the fixed points of $F$ give the general linear group over $\FF_q$:
$$ \GL(n, \bar{\FF}_q)^F = \GL(n, \FF_q).$$
We let $\SL(n, \FF_q)$ denote the special linear group over $\FF_q$, or the elements of determinant $1$ in $\GL(n, \FF_q)$.  The center of $\GL(n, \FF_q)$ is the group of scalar matrices, which is isomorphic to $\FF_q^{\times}$.  We identify the scalar matrices with the group $\FF_q^{\times}$ by a slight abuse of notation.  The projective linear group, which we denote by $\PGL(n, \FF_q)$, is the general linear group modulo its center:
$$ \PGL(n, \FF_q) = \GL(n, \FF_q)/\FF_q^{\times}.$$

Define the map $\tF$ on $\GL(n, \bar{\FF}_q)$ by composing $F$ with the inverse-transpose map, so for $(a_{ij}) \in \GL(n, \bar{\FF}_q)$, we have
 $$\tF((a_{ij})) = {^\top (a_{ij}^q)}^{-1} = (a_{ji}^q)^{-1}.$$
We define the unitary group over $\FF_q$, which we denote by $\U(n, \FF_q)$, to be the group of $\tF$-fixed points in $\GL(n, \bar{\FF}_q)$:
$$ \GL(n, \bar{\FF}_q)^{\tF} = \U(n, \FF_q).$$
Alternatively, one can define $\U(n, \FF_q)$ to be the group of $\tF$-fixed points in $\GL(n, \FF_{q^2})$, which is also the isometry group of the Hermitian form on the vector space $\FF_{q^2}^n$ defined by $\langle v, w \rangle = {^\top v} F(w)$, where $v, w$ are viewed as coordinate vectors and $F$ is the $q$-power map on coordinates.  We identify $\U(1, \FF_q)$ with $(\bar{\FF}_q^{\times})^{\tF}$, which is the multiplicative subgroup of $\FF_{q^2}^{\times}$ of order $q+1$.  Denote this cyclic group by $C_{q+1}$.

The special unitary group $\SU(n, \FF_q)$ is then defined as the group of determinant $1$ elements in $\U(n, \FF_q)$.  The center of $\U(n, \FF_q)$ is the group of scalar matrices with diagonal entries from $C_{q+1}$, and we again identify this group of scalars with $C_{q+1}$.  The projective unitary group $\PGU(n, \FF_q)$ is the unitary group modulo its center, that is,
$$\PGU(n, \FF_q) = \U(n, \FF_q)/C_{q+1}.$$

When $n=0$, we take each of the linear and unitary groups described above to be the group with one element.

\subsection{Polynomials over finite fields} \label{Polynomials}

In this section we define several sets of polynomials over finite fields which we need in order to describe conjugacy classes.  Let $t$ be an indeterminate, and for a finite field $\FF_q$ we let $\FF_q[t]$ denote the collection of polynomials in $t$ with coefficients from $\FF_q$.  We will primarily be interested in monic polynomials with nonzero constant term, and so we denote this collection of polynomials over $\FF_q$ by $M_q[t]$.  

Given a polynomial $f(t) \in M_q[t]$ with $\deg(f(t)) = d$, we define the \emph{reciprocal polynomial} of $f(t)$, denoted by $f^*(t)$, by 
$$ f^*(t) = f(0)^{-1} t^d f(t^{-1}),$$
so if $f(t) = t^d + a_{d-1} t^{d-1} + \cdots + a_1 t + a_0$, then $f^*(t) = t^d + a_1 a_0^{-1} t^{d-1} + \cdots + a_{d-1}a_0^{-1} t + a_0^{-1}$.  A polynomial $f(t) \in M_q[t]$ and its reciprocal $f^*(t)$ have the relationship that $\alpha \in \bar{\FF}_q^{\times}$ is a root of $f(t)$ if and only if $\alpha^{-1}$ is a root of $f^*(t)$ (with the same multiplicity).  A polynomial $f(t) \in M_q[t]$ is called \emph{self-reciprocal} when $f(t) = f^*(t)$, or when $\alpha \in \bar{\FF}_q^{\times}$ is a root of $f(t)$ if and only if $\alpha^{-1}$ is a root with the same multiplicity.  Note that the constant term is $a_0 = \pm 1$ necessarily for a self-reciprocal polynomial.

Now let $r_{q, d}$ denote the number of self-reciprocal polynomials in $M_q[t]$ of degree $d$.  As given in \cite[Lemma 2.1]{GiSi111} and \cite[Lemma 1.3.15(b)]{FNP05}, we have for any prime power $q$ and for $d >0$,
\begin{equation} \label{SelfRecCount}
r_{q,d} = \left\{ \begin{array}{ll} 2 q^{(d-1)/2} & \text{if }q\text{ is odd and }d\text{ is odd,} \\
(q+1)q^{(d/2)-1} & \text{if }q\text{ is odd and }d\text{ is even,} \\
q^{(d-1)/2} & \text{if }q\text{ is even and }d\text{ is odd,}\\
q^{d/2} & \text{if }q\text{ is even and }d\text{ is even.}
\end{array} \right.
\end{equation}
Or, more compactly, if we set $e = e(q) = (2, q-1)$, so $e = 1$ if $q$ is even and $e=2$ if $q$ is odd, then for $d > 0$ we have $r_{q, d} = q^{\lfloor d/2 \rfloor} + (e-1) q^{\lfloor (d-1)/2 \rfloor}$.  For any prime power $q$, we take $r_{q,0} = 1$.

Given any polynomial $f(t) \in M_{q^2}[t]$, define $f^{[q]}(t)$ by applying the $q$-power map (the Frobenius map) to each coefficient of $f(t)$.  So $\alpha \in \bar{\FF}_q^{\times}$ is a root of $f(t)$ if and only if $\alpha^{q}$ is a root of $f^{[q]}(t)$ with the same multiplicity, and $f(t) = f^{[q]}(t)$ if and only if $f(t) \in M_q[t]$.  If $f(t) \in M_{q^2}[t]$ with $\deg(f(t)) = d$, define the \emph{$\sim$-conjugate polynomial} of $f(t)$, written as $\tilde{f}(t)$, by
$$ \tilde{f}(t) = f(0)^{-q} t^d f^{[q]}(t^{-1}).$$
So if $f(t) = t^d + a_{d-1} t^{d-1} + \cdots + a_1 t + a_0$, then $\tilde{f}(t) = t^d + (a_1 a_0^{-1})^q t^{d-1} + \cdots + (a_{d-1} a_0^{-1})^q t + a_0^{-q}$.  In particular, $\alpha \in \bar{\FF}_q^{\times}$ is a root of $f(t)$ if and only if $\alpha^{-q}$ is a root of $\tilde{f}(t)$ with the same multiplicity.  A polynomial $f(t) \in M_{q^2}[t]$ is \emph{self-conjugate} if $f(t) = \tilde{f}(t)$.  Define $\cU_q[t]$ to be the collection of self-conjugate polynomials in $M_{q^2}[t]$, so
$$ \cU_q[t] = \{ f(t) \in M_{q^2}[t] \, \mid \, f(t) = \tilde{f}(t) \}.$$

Now consider some $f(t) \in \cU_q[t]$ which is also self-reciprocal.  Then for any $\alpha \in \bar{\FF}_q^{\times}$, we have $\alpha$ is a root of $f(t)$ if and only if $\alpha^{-1}$ is, if and only if $\alpha^{-q}$ is, all of the same multiplicity.  But then $\alpha^{-q}$ is a root if and only if $\alpha^q$ is, since $f(t)$ is self-reciprocal, which implies $\alpha$ is a root of $f(t)$ if and only if $\alpha^q$ is (of the same multiplicity), and it follows that $f(t)$ must be a self-reciprocal polynomial in $M_q[t]$.  That is, we have
\begin{equation} \label{SelfDualSets}
\left\{ f(t) \in \cU_q[t] \, \mid \, f(t) = f^*(t) \right\} = \left\{ f(t) \in M_q[t] \, \mid \, f(t)=f^*(t) \right\} = \cU_q[t] \cap M_q[t],
\end{equation}
and the number of polynomials of degree $d$ in this set is equal to $r_{q,d}$ given by \eqref{SelfRecCount}.

Now let $f(t) \in M_q[t] \cup \cU_q[t]$, with $\zeta \in \FF_q^{\times}$ if $f(t) \in M_q[t]$, and $\zeta \in C_{q+1}$ if $f(t) \in \cU_q[t]$.  With a fixed $\zeta$ and $\deg(f(t)) = d$, we define the \emph{$\zeta$-reciprocal polynomial} of $f(t)$, written as $\hat{f}(t)$, by 
$$\hat{f}(t) = f(0)^{-1} t^d f(\zeta t^{-1}),$$
so if $f(t) = t^d + a_{d-1} t^{d-1} + \cdots + a_1 t + a_0$, then $\hat{f}(t) = t^d + a_1\zeta a_0^{-1} t^{d-1} + \cdots + a_{d-1} \zeta^{d-1} a_0^{-1} t + \zeta^d a_0^{-1}$.  The polynomial $f(t)$ is \emph{$\zeta$-self-reciprocal} if $f(t) = \hat{f}(t)$, which is equivalent to the statement that $\alpha \in \bar{\FF}_q^{\times}$ is a root of $f(t)$ if and only if $\zeta \alpha^{-1}$ is a root of the same multiplicity.

We will be interested in $\zeta$-reciprocal polynomials in the case that $\zeta$ is not a square in $\FF_q^{\times}$ or $C_{q+1}$, respectively, which means we will only be concerned in the case that $q$ is odd.  If $\zeta$ is not a square in $\FF_q^{\times}$, let $r^{\zeta}_{q,d}$ denote the number of $\zeta$-self-reciprocal polynomials in $M_q[t]$ of degree $d$.  Gill and Singh \cite[Lemma 2.2]{GiSi111} prove that
\begin{equation} \label{ZetaSelfCount}
r^{\zeta}_{q,d} = \left\{ \begin{array}{ll} r_{q, d} & \text{if } d \text{ is even,} \\
0 & \text{if } d \text{ is odd.} 
\end{array}\right.
\end{equation}
We need the following analogue of this statement for $\zeta$-self-reciprocal polynomials in $\cU_q[t]$.

\begin{lemma} \label{ZetaUSelfCount}
Let $q$ be odd and $\zeta \in C_{q+1}$ be a non-square.  The number of $\zeta$-self-reciprocal polynomials in $\cU_q[t]$ of degree $d$ is equal to $r^{\zeta}_{q,d}$, given in \eqref{ZetaSelfCount}.
\end{lemma}
\begin{proof}  Let $f(t) \in \cU_q[t]$ be $\zeta$-self-reciprocal of degree $d$, with $f(t) = t^d + a_{d-1} t^{d-1} + \cdots + a_1 t + a_0$.  Since $\hat{f}(t) = f(t)$, then we have $a_0 = \zeta^d a_0^{-1}$, so $a_0^2 = \zeta^d$.  Since also $\tilde{f}(t) = f(t)$, we have $a_0^{-q} = a_0$, so that $a_0 \in C_{q+1}$.  Since $\zeta$ is a non-square in $C_{q+1}$, then this is impossible with $d$ odd, and so there are no such polynomials in this case.

We now assume $d$ is even, and from $a_0^2 = \zeta^d$ we have $a_0 = \pm \zeta^{d/2}$, and for any $0 \leq i \leq d$ we have $a_0 \zeta^{i-d} = a_0^{-1} \zeta^i$.  From the fact that $\hat{f}(t) = f(t)$ and $\tilde{f}(t) = f(t)$, we know that for $0 < i < d$,
\begin{equation} \label{coeffs}
a_i = a_0 a_{d-i} \zeta^{-i} \quad \text{and} \quad a_{d-i} = a_i^q a_0^{-q} = a_i^q a_0.
\end{equation}
Note that if $a_i \in \FF_{q^2}$ is chosen to satisfy the equations above for $0 < i \leq d/2$, then $a_{d-i}$ is determined for $d/2 < i < d$.  Substituting the second equation of \eqref{coeffs} into the first yields
$$ a_i = a_0^2 \zeta^{-i} a_i^q = \zeta^{d-i} a_i^q.$$
For any $i$, $0 < i < d$, there are $q$ solutions to this equation in $\FF_{q^2}$, given by either $a_i = 0$, or the $q-1$ solutions to $a_i^{q-1} = \zeta^{i-d}$ if $a_i \neq 0$.

Suppose first that $a_0 = -\zeta^{d/2}$.  Then for $i = d/2$, the first equation in \eqref{coeffs} gives $a_{d/2} = -a_{d/2}$, so that $a_{d/2}=0$ necessarily.  Given that there are $q$ possibilities for each of $a_i$ for $0 < i < d/2$, we have a total of $q^{(d/2) - 1}$ polynomials in this case.

If $a_0 = \zeta^{d/2}$, then there is no such restriction on $a_{d/2}$, and there are $q$ possibilities for its value.  Taking the $q$ possible values for $a_i$ with $0 < i < d/2$, the polynomial $f(t)$ is determined, and there are $q^{d/2}$ possibilites.  This gives a total of $q^{(d/2)- 1} + q^{d/2} = r_{q,d}$ polynomials of degree $d$ which are $\zeta$-self-reciprocal in $\cU_q[t]$ when $d$ is even, as claimed.
\end{proof}

\section{Real Conjugacy Classes} \label{RealClasses}

\subsection{Conjugacy classes and real classes in $\GL(n, \FF_q)$ and $\U(n, \FF_q)$} \label{RealClassesSLSU}

The conjugacy classes of $\GL(n,\FF_q)$ may be parameterized by sequences of polynomials, 
\begin{equation} \label{GLclassparam}
 (f_1(t), f_2(t), \ldots), \text{ with } \, f_i(t) \in M_q[t] \text{ such that } \, \sum_{i \geq 1} i\deg(f_i(t)) = n,
\end{equation}
as explained by Macdonald \cite[Section 1]{Ma81}.  In fact, Macdonald uses sequences of polynomials with constant term $1$ instead of monic polynomials.  This may be seen to be equivalent to \eqref{GLclassparam} by replacing the polynomial $f_i(t) = \prod_{j=1}^d (t - \alpha_j)$ with the polynomial $\prod_{j=1}^d (1 - t \alpha_j)$.  In the parametrization \eqref{GLclassparam}, for any element $g \in \GL(n, \FF_q)$ in the conjugacy class corresponding to the sequence $(f_1(t), f_2(t), \ldots)$, the characteristic polynomial of $g$ is given by $\prod_{i \geq 1} f_i(t)^i$.

As given in \cite[Section 6]{Ma81}, the conjugacy classes of $\U(n, \FF_q)$ may be similarly parameterized by sequences of polynomials
\begin{equation} \label{Uclassparam}
(f_1(t), f_2(t), \ldots), \text{ with } \, f_i(t) \in \cU_q[t], \text{ such that } \, \sum_{i \geq 1} i\deg(f_i(t)) = n,
\end{equation}
where the characteristic polynomial of any element in this class is given by $\prod_{i \geq 1} f_i(t)^i$.

For the conjugacy class of $\GL(n, \FF_q)$ or $\U(n, \FF_q)$ corresponding to the sequence of polynomials $(f_1(t), f_2(t), \ldots)$, we may define the partition 
\begin{equation} \label{PartitionClass}
\nu = (1^{m_1} 2^{m_2} 3^{m_3} \cdots), \text{ where } \, \deg(f_i(t)) = m_i,
\end{equation}
corresponding to this conjugacy class, where $\nu$ is a partition of $n$.  Fixing a partition $\nu \in \cP_n$, a conjugacy class of $\GL(n, \FF_q)$ or $\U(n, \FF_q)$ is said to be a conjugacy class of \emph{type $\nu$} if the class corresponds to the partition $\nu$ as given by \eqref{PartitionClass}.

An element $g$ of a group $G$ is said to be \emph{real} if $g$ is conjugate to $g^{-1}$ in $G$.  If the element $g$ is real, then all elements in the conjugacy class of $g$ are real, in which case we call this a \emph{real conjugacy class} of $G$.  A conjugacy class of $\GL(n, \FF_q)$ corresponding to $(f_1(t), f_2(t), \ldots)$ is a real class if and only if each $f_i(t)$ is self-reciprocal \cite[Proposition 3.7]{GiSi111}.  The same statement is true for real conjugacy classes of $\U(n, \FF_q)$ parameterized by \eqref{Uclassparam}, as explained in \cite[Section 5.2]{GaSiVi14}.  Since the set of self-reciprocal polynomials in $M_q[t]$ is the same as the set of self-reciprocal polynomials in $\cU_q[t]$ as in \eqref{SelfDualSets}, then the real classes of $\GL(n, \FF_q)$ and of $\U(n, \FF_q)$ may be parameterized by exactly the same sequences of polynomials, a fact which reflects the observation of Gow \cite[pg. 181]{Go81} that these classes are equal in number.  

Let $r_{q, d}$ be the number of self-reciprocal polynomials in $M_q[t]$ (or in $\cU_q[t]$) of degree $d$, as given in \eqref{SelfRecCount}.  By considering the number of real conjugacy classes of type $\nu$ for each partition $\nu$, where $\nu = (1^{m_1} 2^{m_2} 3^{m_3} \cdots)$, the number of real classes in $\GL(n, \FF_q)$ or in $\U(n, \FF_q)$ is given by the coefficient of $u^n$ in the generating function 
\begin{equation} \label{RealClassesGen}
\sum_{n \geq 0} \left( \sum_{\nu \in \cP_n} \prod_{i: m_i >0} r_{q, m_i} \right) u^n = \prod_{i = 1}^{\infty} \left( \sum_{k \geq 0} (u^i)^k r_{q, k} \right) = \prod_{i =1}^{\infty} \frac{(1+u^i)^e}{1-qu^{2i}},
\end{equation}
where $e = e(q) = (q-1, 2)$ (see \cite[Theorem 3.8]{GiSi111} and \cite[Theorem 2.9]{Go81}).

Next consider those real classes of $\GL(n, \FF_q)$ or $\U(n, \FF_q)$ which are contained in $\SL(n, \FF_q)$ or $\SU(n, \FF_q)$, respectively.  Since an element $g$ of the conjugacy class parameterized by the sequence $(f_1(t), f_2(t), \ldots)$ has characteristic polynomial $\prod_{i \geq 1} f_i(t)^i$, which has constant term $(-1)^n \det(g)$, then elements of this class have determinant $1$ exactly when this constant term is $(-1)^n$.  That is, a real class of $\GL(n, \FF_q)$ or of $\U(n, \FF_q)$ which is contained in $\SL(n, \FF_q)$ or in $\SU(n, \FF_q)$, respectively, corresponds to a sequence $(f_1(t), f_2(t), \ldots)$ of self-reciprocal polynomials such that $\prod_{i \geq 1} f_i(0)^i = (-1)^n$.  In particular, we have the following observation.

\begin{lemma} \label{GLSLGUSU} Let $n \geq 1$ and let $q$ be any prime power.  Then the number of real classes of $\GL(n, \FF_q)$ contained in $\SL(n, \FF_q)$ is equal to the number of real classes of $\U(n, \FF_q)$ contained in $\SU(n, \FF_q)$.
\end{lemma}

For any partition $\nu$ of $n$, let $sl_{\nu}$ denote the number of real classes of type $\nu$ in $\GL(n, \FF_q)$ which are contained in $\SL(n, \FF_q)$.  Gill and Singh \cite[Proposition 4.1]{GiSi111} compute $sl_{\nu}$ to be
\begin{equation} \label{slnucount}
sl_{\nu} = \left\{ \begin{array}{ll} \displaystyle\prod_{i: m_i > 0} r_{q, m_i} & \text{if } q \text{ is even or } m_i = 0 \text{ for } i \text{ odd,} \\
\displaystyle\frac{1}{2} \prod_{i: m_i > 0} r_{q, m_i} & \text{if } q \text{ is odd or } im_i \text{ is odd for some }i, \\
\displaystyle h_{\nu}(q) \prod_{i \text{ odd}: \atop{m_i > 0}} q^{(m_i/2)-1} \prod_{i \text{ even}:\atop{m_i > 0}} r_{q, m_i} & \text{otherwise},
\end{array} \right.
\end{equation}
with $h_{\nu}(q) = \frac{1}{2} ( (q+1)^{\rho} + (q-1)^{\rho})$, where $\rho$ is the number of odd $i$ such that $m_i > 0$.  We simplify this expression a bit as follows.  

\begin{lemma} \label{SLnuRevised}
If $q$ is even, then the number of real conjugacy classes of $\GL(n, \FF_q)$ is equal to the number of real conjugacy classes of $\GL(n, \FF_q)$ contained in $\SL(n, \FF_q)$.  If $q$ is odd, the number of real conjugacy classes of $\GL(n, \FF_q)$ of type $\nu$ which are contained in $\SL(n, \FF_q)$ is given by
$$sl_{\nu} = \left\{ \begin{array}{ll} \displaystyle \frac{1}{2} \prod_{i: m_i > 0} r_{q, m_i} & \text{ if } m_i \text{ is odd for some odd } i, \\
\displaystyle \frac{1}{2} \left( \prod_{i: m_i > 0} r_{q, m_i} + \prod_{i \text{ odd}: \atop{ m_i > 0}} \frac{q-1}{q+1} r_{q, m_i} \prod_{i \text{ even}: \atop{ m_i > 0}} r_{q, m_i} \right) & \text{ if } m_i \text{ is even for all odd } i. \end{array} \right.$$
\end{lemma}
\begin{proof}  First, if $q$ is even, this statement follows from the first case of \eqref{slnucount}.  Note also this follows from the fact that the determinant of any real element of $\GL(n, \FF_q)$ must be $\pm 1$, and so must be in $\SL(n, \FF_q)$ in the case $q$ is even.

Now suppose $q$ is odd, and we are not in the second case of \eqref{slnucount}, so that $m_i$ is even whenever $i$ is odd.  First note that in the case that $m_i = 0$ for all odd $i$, we may interpret the expression in the first case of \eqref{slnucount} as the expression in the third case, with $\rho = 0$ and the product over all $i$ odd with $m_i > 0$ is empty.  That is, the first and third cases in \eqref{slnucount} may be combined.  Now consider the expression in the third case, so that $m_i$ is even whenever $i$ is odd, and note that $\rho$ is exactly the number of contributing factors in the product over odd $i$ with $m_i > 0$.  That is, 
\begin{align*}
h_{\nu}(q) \prod_{i \text{ odd}: \atop{m_i > 0}} q^{(m_i/2) - 1} & = \frac{1}{2} \left(\prod_{i \text{ odd}: \atop{m_i > 0}} (q+1) q^{(m_i/2) - 1} + \prod_{i \text{ odd}: \atop{m_i > 0}} (q-1) q^{(m_i/2) - 1} \right) \\
& = \frac{1}{2} \left( \prod_{i \text{ odd}: \atop{m_i > 0}}  r_{q, m_i} + \prod_{i \text{ odd}: \atop{m_i > 0}}\frac{q-1}{q+1} r_{q, m_i} \right),
\end{align*}
by applying \eqref{SelfRecCount} and the fact that $m_i$ is even when $i$ is odd in this case.  Substituting this expression in for the third case of \eqref{slnucount} and combining the first and third cases gives the claimed expression.
\end{proof}

\subsection{Real classes of $\PGL(n, \FF_q)$ and $\PGU(n, \FF_q)$} \label{RealClassesPGLPGU}

We begin by describing the conjugacy classes of $\PGL(n, \FF_q)$ and $\PGU(n, \FF_q)$, following \cite[Sections 2 and 6]{Ma81}.  Let $Z$ be the center of $\GL(n, \FF_q)$ or $\U(n, \FF_q)$, identified with $\FF_q^{\times}$ or $C_{q+1}$, respectively.  Let $G$ be either $\GL(n, \FF_q)$ or $\U(n, \FF_q)$, and let $\bar{G} = G/Z$ be either $\PGL(n, \FF_q)$ or $\PGU(n, \FF_q)$, respectively.  If $xZ$ and $yZ$ are elements of $\bar{G}$, then $xZ$ and $yZ$ are conjugate in $\bar{G}$ if and only if $x$ and $\eta y$ are conjugate in $G$ for some $\eta \in Z$.  If the conjugacy class of $y$ in $G$ corresponds to the sequence $(f_1(t), f_2(t), \ldots)$ as in Section \ref{RealClassesSLSU}, where $\deg(f_i(t)) = d_i$, then the conjugacy class of $\eta y$ corresponds to the sequence $(\eta^{d_1} f_1(t \eta^{-1}), \eta^{d_2} f_2(t\eta^{-1}), \ldots)$,
since $\alpha$ is a root of $f_i(t)$ if and only if $\eta \alpha$ is a root of the monic polynomial $\eta^{d_i} f_i(t\eta^{-1})$.  So, we define the action of $Z$ on polynomials $f(t)$ and on sequences $(f_i(t)) = (f_1(t), f_2(t), \dots)$ (with $d = \deg(f(t))$) by
\begin{equation} \label{Zaction}
\eta.f(t) = \eta^d f(t \eta^{-1}),  \text{ and }  \eta.(f_i(t)) = (\eta.f_i(t)) = (\eta.f_1(t), \eta.f_2(t), \ldots).
\end{equation}
Then the conjugacy classes of $\bar{G}$ are parameterized by the orbits of the $Z$-action on the sequences $(f_i(t))$.  Note that if the conjugacy class of $G$ corresponding to $(f_1(t), f_2(t), \ldots)$ is of type $\nu$, then so is the class corresponding to $\eta.(f_1(t), f_2(t), \ldots)$.  So we say a conjugacy class of $\bar{G}$ is type $\nu$ if it corresponds to a $Z$-orbit of classes in $G$ which are of type $\nu$.

The real classes of $\PGL(n, \FF_q)$ were described by Gill and Singh \cite[Section 2]{GiSi112}, and here we follow their work closely to describe the real classes of $\PGU(n, \FF_q)$.  For the rest of this section we fix some non-square $\zeta \in Z$.  An element $g \in G$ is \emph{$\zeta$-real} if $g$ is conjugate to $\zeta g^{-1}$ in $G$.  If $xZ \in \bar{G}$, we say $xZ$ \emph{lifts} to the element $g \in G$ if $g \in xZ$.  The following is \cite[Lemma 2.4]{GiSi112} in the case $G = \GL(n, \FF_q)$, and the proof is exactly the same in the case $G = \U(n, \FF_q)$.

\begin{lemma} \label{RealLift}
If $gZ \in \bar{G}$ is real in $\bar{G}$, then $gZ$ lifts to real or a $\zeta$-real element in $G$.
\end{lemma}

A conjugacy class of $G$ corresponding to $(f_1(t), f_2(t), \ldots)$ consists of $\zeta$-real elements if and only if each $f_i(t)$ is a $\zeta$-self-reciprocal polynomial.  Thus, by Lemma \ref{RealLift} in order to understand the real conjugacy classes of $\bar{G}$, we must understand the $Z$-orbits of sequences $(f_1(t), f_2(t), \ldots)$, where every $f_i(t)$ is self-reciprocal or every $f_i(t)$ is $\zeta$-self-reciprocal.  This essentially requires the understanding of $Z$-orbits of individual self-reciprocal or $\zeta$-self-reciprocal polynomials.  

For these purposes, we define the following notation (following \cite{GiSi112}).  Let $T_d$ denote the set of self-reciprocal polynomials of degree $d$ which are in $\cU_q[t]$, and recall from Section \ref{Polynomials} that this is the same set whether we are in $M_q[t]$ or in $\cU_q[t]$.  We let $S_d$ denote the set of $\zeta$-self-reciprocal polynomials of degree $d$ which are in $M_q[t]$ (in the $G = \GL(n, \FF_q)$ case) or in $\cU_q[t]$ (in the $G = \U(n, \FF_q)$ case).  While these are in general distinct sets of polynomials in these two cases, they do have the same cardinality by Lemma \ref{ZetaUSelfCount}.  Given any $f(t)$ in $M_q[t]$ or $\cU_q[t]$, we let $[f]$ denote the $Z$-orbit of $f(t)$, and we let $[f]_T = [f] \cap T_d$ and $[f]_S = [f] \cap S_d$ when $\deg(f(t)) = d$.  In the rest of this section, we follow the same arguments for $G = \U(n, \FF_q)$ as are given for $G = \GL(n, \FF_q)$ in \cite[Section 2]{GiSi112}.  Since many of the details are essentially the same, we will give outlines of proofs with mostly details which are complementary to those given in \cite[Section 2]{GiSi112}.

The following is the $G = \U(n, \FF_q)$ version of \cite[Lemma 2.2]{GiSi112}.

\begin{lemma} \label{OrbitSize} Let $f(t) \in \cU_q[t]$, with $f(t) = t^d + a_{d-1}t^{d-1} + \cdots + a_1 t + a_0$. 
\begin{enumerate}
\item[(i)] If $q$ is even, then $[f]_T$ contains at most one element (and $[f]_S$ is empty).
\item[(ii)] If $q$ is odd, then $[f]_S$ and $[f]_T$ contain at most $2$ elements.  In particular, $[f]_T$ and $[f]_S$ (when nonempty) may be assumed to be of the form $\{ f(t), \eta.f(t) \}$ where $\eta \in C_{q+1}$ has order a power of $2$.
\end{enumerate}
\end{lemma}
\begin{proof}  Let $\eta \in C_{q+1}$.  If $f(t)$ and $\eta.f(t)$ are both in $T_d$ or both in $S_d$, one obtains $\eta^i a_i = \eta^{-i} a_i$ for $0 < i < d$.  That is, $|\eta|$ must divide $2k$ whenever $a_k \neq 0$.

Suppose $|\eta|$ is odd.  If both $f(t)$ and $\eta.f(t)$ are in $[f]_S$ or both in $[f]_T$, then $|\eta|$ divides $k$ whenever $a_k \neq 0$, and so $f(t) \in \cU_q[t^{|\eta|}]$.  In particular, $\eta.f(t) = f(t)$.  If $q$ is even, then $|\eta|$ divides $q+1$ and so must be odd, and the result follows in this case.

If $q$ is odd, suppose that for some $\beta, \eta \in C_{q+1}$ of even order we have $f(t)$, $\eta.f(t)$, and $\beta.f(t)$ are all in $T_d$ or all in $S_d$, and that $\beta.f(t)$ and $\eta.f(t)$ are distinct from $f(t)$.  Thus $f(t) \not\in \cU_q[t^{|\eta|}] \cup \cU_q[t^{|\beta|}]$, while $|\beta|/2$ and $|\eta|/2$ both divide $k$ whenever $a_k \neq 0$, so $f(t)$ lies in $\cU_q[t^{|\eta|/2}] \cap \cU_q[t^{|\beta|/2}] = \cU_q[t^{\mathrm{lcm}(|\eta|/2, |\beta|/2)}]$.  It follows that we must have $|\eta|_2 = |\beta|_2$, and then that $\eta.f(t) = \beta.f(t)$.  Note also that if $|\eta| = 2^k s$ with $s$ odd, and $\gamma = \eta^{2^k}$, then $\gamma$ has odd order so $\gamma.f(t) = f(t)$.  Then $|\eta \gamma^{-1}| = 2^k$ and $\eta\gamma^{-1}.f(t) = \eta.f(t)$, and the result follows.
\end{proof}

The following is very similar to statements contained in \cite[Proof of Lemma 2.3]{GiSi112}.  

\begin{lemma} \label{STOrbit}
Let $q$ be odd, $\eta \in C_{q+1}$, and $f(t) = t^d + a_{d-1}t^{d-1} + \cdots + a_1 t + a_0 \in \cU_q[t]$ with $d$ even.
\begin{enumerate}
\item[(i)] If $f(t) \in S_d$, then $\eta.f(t) \in T_d$ if and only if $f(t) \in \cU_q[t^{|q+1|_2}]$ and $\eta^2 = \zeta\beta^{-1}$ for some $\beta \in C_{q+1}$ with $|\beta|_2 = |q+1|_2$.
\item[(ii)] If $f(t) \in T_d$, then $\eta.f(t) \in S_d$ if and only if $f(t) \in \cU_q[t^{|q+1|_2}]$ and $\eta^2 = \zeta \beta$ for some $\beta \in C_{q+1}$ with $|\beta|_2 = |q+1|_2$.
\end{enumerate}
\end{lemma}
\begin{proof}  The proofs of (i) and (ii) are almost identical, so we give an outline of (i).  If $f(t) \in S_d$, then from the proof of Lemma \ref{ZetaUSelfCount}, we have $a_{d-i} = a_i a_0 \zeta^{i-d}$ for $0 < i < d$.  If we also have $\eta.f(t)$ is self-reciprocal, then we can compute that $(\eta^{2})^{d-i} a_i= \zeta^{i-d} a_i$.  Since $a_i \neq 0$ if and only if $a_{d-i} \neq 0$, then we have $\eta^{2i} = \zeta^{-i}$ whenever $a_i \neq 0$ (and note $\eta^{2d} = \zeta^{-d}$ since $a_0^2 = \zeta^d$).  Thus $\eta^2 = \beta \zeta^{-1}$ for some $\beta \in C_{q+1}$, where $|\beta|$ divides the greatest common divisor of all $i \neq 0, d$ such that $a_{i} \neq 0$.  Since $\zeta$ is a non-square, then $\beta$ is also a non-square, and so $|q+1|_2 = |\beta|_2$.  Thus $|q+1|_2$ divides all $i$ such that $a_i \neq 0$, and so $f(t) \in \cU_q[t^{|q+1|_2}]$.
\end{proof}

The following result is the unitary analog of \cite[Lemma 2.3]{GiSi112}.

\begin{lemma} \label{OrbitSize2}  Let $q$ be odd, and let $f(t) \in \cU_q[t]$ with $f(t) = t^d + a_{d-1} t^{d-1} + \cdots + a_1 t + a_0$.
\begin{enumerate}
\item[(i)] If $d$ is odd, then $S_d$ is empty, and if $f(t) \in T_d$ then $|[f]_{T}| = 2$.
\item[(ii)] Suppose $d$ is even and $f(t) \in \cU_q[t^{|q+1|_2}]$.  If $f(t) \in T_d$ or $f(t) \in S_d$, then $|[f]_S| = |[f]_T| = 1$.
\item[(iii)] Suppose $d$ is even and $f(t) \not\in \cU_q[t^{|q+1|_2}]$.  If $f(t) \in S_d$, then $|[f]_S| = 2$ and $|[f]_T| = 0$.  If $f(t) \in T_d$, then $|[f]_T| = 2$ and $|[f]_S| = 0$.

\end{enumerate}
\end{lemma}
\begin{proof} If $d$ is odd, we have already mentioned in Lemma \ref{ZetaUSelfCount} that $f(t)$ cannot be $\zeta$-self-reciprocal, and so $S_d$ is empty.  If $f(t) \in T_d$, then we know $[f]_T$ contains at most two elements by Lemma \ref{OrbitSize}(ii).  But $(-1).f(t) = -f(-t) \neq f(t)$ since $d$ is odd, and so $[f]_T = \{ f(t), (-1).f(t) \}$ has two elements in this case, and (i) follows.

Now suppose $d$ is even and that $f(t) \in S_d$ or $f(t) \in T_d$.  By Lemma \ref{OrbitSize}, if $f(t) \in S_d$ (or $f(t) \in T_d$, respectively), then we may assume $[f]_{S}$ (or $[f]_T$) is of the form $\{ f(t), \eta.f(t) \}$ for some $\eta \in C_{q+1}$ with order a power of $2$.  Suppose $f(t) \in \cU_q[t^{|q+1|_2}]$.  It follows that $\eta.f(t) = f(t)$, so $[f]_S = \{ f(t) \}$ when $f(t) \in S_d$ (and $[f]_T = \{ f(t) \}$ when $f(t) \in T_d$).  It follows directly from Lemma \ref{STOrbit} that if $f(t) \in S_d$, then there is an $\eta \in C_{q+1}$ such that $\eta.f(t) \in T_d$ (and if $f(t) \in T_d$ then $\eta.f(t) \in S_d$ for some $\eta \in C_{q+1}$).  Thus $|[f]_S| = |[f]_T| = 1$ in all cases, and (ii) follows.

Finally, suppose $f(t) \not\in \cU_q[t^{[q+1]_2}]$, and let $b$ be the smallest power of $2$ such that $f(t) \not\in \cU_q[t^b]$.  Taking $\eta \in C_{q+1}$ such that $|\eta|=b$, we have $\eta.f(t) \neq f(t)$, and if $f(t) \in S_d$ (or $f(t) \in T_d$), then also $\eta.f(t) \in S_d$ (or $\eta.f(t) \in T_d$).  So by Lemma \ref{OrbitSize}(ii), we have $|[f]_S| = 2$ (or $|[f]_T = 2$).  It also follows from Lemma \ref{STOrbit} that if $f(t) \in S_d$, then $|[f]_T| =0$ and if $f(t) \in T_d$, then $|[f]_S| = 0$.  Thus we have (iii).
\end{proof}

We are now able to classify the real classes of $\PGU(n, \FF_q)$ in the following, which is analogous with \cite[Lemma 2.6]{GiSi112}.

\begin{lemma} \label{PGUReal}
Let $q$ be odd.  Consider a conjugacy class of type $\nu = (1^{m_1} 2^{m_2} \cdots)$ in $\U(n, \FF_q)$ parameterized by the sequence $(f_1(t), f_2(t), \ldots)$, where $\deg(f_i(t)) = m_i$.  Let $[(f_i(t))]$ denote the $Z$-orbit of this sequence.
\begin{enumerate}
\item[(i)] If some $m_i$ is odd, then $[(f_i(t))]$ contains no $\zeta$-real classes, and contains either zero or two real classes.
\item[(ii)] If all $m_i$ are even and $f_i(t) \in \cU_q[t^{|q+1|_2}]$ for all $i$, then either $[(f_i(t))]$ contains no real or $\zeta$-real classes, or contains exactly one real class and exactly one $\zeta$-real class.
\item[(iii)] If all $m_i$ are even and $f_j(t) \not\in \cU_q[t^{|q+1|_2}]$ for some $j$, then $[(f_i(t))]$ contains either no real or $\zeta$-real classes, or exactly two real classes and no $\zeta$-real classes, or exactly two $\zeta$-real classes and no real classes.
\end{enumerate}
\end{lemma}
\begin{proof} If $m_j = \deg(f_j(t))$ is odd, then $\eta.f_j(t)$ is never $\zeta$-self-reciprocal for any $\eta \in C_{q+1}$.  Thus $[(f_i(t))]$ cannot contain any $\zeta$-real classes.  Suppose $[(f_i(t))]$ contains a real class, and without loss of generality suppose $(f_i(t))$ is a real class, so that every $f_i(t)$ is self-reciprocal.  By Lemma \ref{OrbitSize}(ii), there can be at most one other real class in $[(f_i(t))]$.  As in Lemma \ref{OrbitSize2}(i), we have $(-1).f_i(t)$ is self-reciprocal for each $i$.  Since $m_j = \deg(f_j(t))$ is odd, then $(-1).f_j(t) \neq f_j(t)$, and so $(-1).(f_i(t)) \neq (f_i(t))$.  Thus $[(f_i(t))]$ contains two real classes.

Now suppose all $m_i$ are even and all $f_i(t) \in \cU_q[t^{|q+1|_2}]$.  If $(f_i(t))$ is a $\zeta$-real class (or a real class, respectively), we may choose $\eta \in C_{q+1}$ as in Lemma \ref{STOrbit}(i) (Lemma \ref{STOrbit}(ii), respectively) so that $\eta.(f_i(t))$ is a real class (a $\zeta$-real class, respectively).  Now (ii) follows from Lemma \ref{OrbitSize2}(ii).

Finally, suppose all $m_i$ are even and some $f_j(t) \not\in \cU_q[t^{|q+1|_2}]$.  Note that from Lemma \ref{STOrbit}, if $[(f_i(t))]$ contains a real class, then it cannot contain a $\zeta$-real class and vice versa.  Now let $b$ be the smallest power of $2$ such that all $f_i(t) \not\in \cU_q[t^{b}]$, and let $\eta \in C_{q+1}$ such that $|\eta|=b$.  As in the proof of Lemma \ref{OrbitSize2}(iii), if $(f_i(t))$ is a real class (or a $\zeta$-real class), then $\eta.(f_i(t))$ is a distinct real class (or a $\zeta$-real class).  The statement now follows from Lemma \ref{OrbitSize2}(iii).
\end{proof}

We may now give the following enumeration of real conjugacy classes in $\PGU(n, \FF_q)$.

\begin{lemma} \label{PGUPGLReal}
Let $\nu = (1^{m_1} 2^{m_2} \cdots )$ be a partition of $n$, and let $pgu_{\nu}$ be the number of real conjugacy classes of $\PGU(n, \FF_q)$ of type $\nu$.  Then we have
$$pgu_{\nu} = \left\{ \begin{array}{ll} \displaystyle\prod_{i: m_i > 0} r_{q, m_i} & \text{ if } q \text{ is even, or if } q \text{ is odd and } m_i \text{ is even for all } i, \\ \displaystyle\frac{1}{2} \prod_{i: m_i > 0} r_{q, m_i} & \text{ if } q \text{ is odd and } m_i \text{ is odd for some } i. \end{array} \right.$$
Moreover, if $pgl_{\nu}$ is the number of real classes of $\PGL(n, \FF_q)$ of type $\nu$, then $pgl_{\nu} = pgu_{\nu}$.  In particular, the number of real classes of $\PGL(n, \FF_q)$ is equal to the number of real classes of $\PGU(n, \FF_q)$.
\end{lemma}
\begin{proof} Let $(f_i(t)) = (f_1(t), f_2(t), \ldots)$ correspond to a conjugacy class of type $\nu$ in $\U(n, \FF_q)$.  First, if $q$ is even, then it follows from Lemma \ref{OrbitSize}(i) that the $Z$-orbit $[(f_i(t))]$ can contain at most one real class and no $\zeta$-real classes.  It follows that the real classes of $\PGU(n, \FF_q)$ of type $\nu$ are in bijection with the real classes of $\U(n, \FF_q)$ of type $\nu$.  

We now assume that $q$ is odd.  If some $m_i = \deg(f_i(t))$ is odd, then by Lemma \ref{PGUReal}(i), the $Z$-orbit $[(f_i(t))]$ contains no $\zeta$-real classes, and either contains no or two real classes.  Thus the real classes of $\PGU(n, \FF_q)$ of type $\nu$ correspond to pairs of real classes of $\U(n, \FF_q)$ of type $\nu$, and so there are half as many real classes of $\PGU(n, \FF_q)$ of type $\nu$ as there are of $\U(n, \FF_q)$.

Finally, suppose every $m_i = \deg(f_i(t))$ is even, and assume the $Z$-orbit $[(f_i(t))]$ corresponds to a real class of $\PGU(n, \FF_q)$, and so contains either a real or a $\zeta$-real class of $\U(n, \FF_q)$.   Then by Lemma \ref{PGUReal}(ii) and (iii), $[(f_i(t))]$ contains either two real classes of $\U(n, \FF_q)$, two $\zeta$-real classes, or one of each.  By Lemma \ref{ZetaUSelfCount}, the number of $\zeta$-self-dual polynomials of degree $m_i$ in $\cU_q[t]$ is equal to the number of self-dual polynomials of degree $m_i$ in $\cU_q[t]$ (since $m_i$ is even), which is $r_{q, m_i}$.  So there are an equal number of $\zeta$-real classes and real classes of type $\nu$ in $\PGU(n, \FF_q)$ in this case, each of which are given by $\prod_{i: m_i > 0} r_{q, m_i}$.  Since sequences of such polynomials are paired to form the real classes of $\PGU(n, \FF_q)$ of type $\nu$, the result follows.  

In all cases, this matches the number of real classes in $\PGL(n, \FF_q)$ of type $\nu$ obtained by Gill and Singh in \cite[Corollary 2.7 and Theorem 2.8]{GiSi112}, and so the number of real classes of $\PGL(n, \FF_q)$ is equal to the number of real classes of $\PGU(n, \FF_q)$.
\end{proof}

\section{Main Result} \label{MainResult}

We finally arrive at our main result.

\begin{theorem} \label{main}  Let $q$ be any prime power.  Then for any $n \geq 0$ we have
\begin{align*}
& \text{ the number of real classes in } \PGL(n, \FF_q) \\
=&\text{ the number of real classes of } \GL(n, \FF_q) \text{ contained in } \SL(n, \FF_q) \\
=&\text{ the number of real classes in } \PGU(n, \FF_q) \\
=&\text{ the number of real classes of } \U(n, \FF_q) \text{ contained in } \SU(n, \FF_q).
\end{align*}
If we take $e = e(q) = (q-1, 2)$, then the generating function for this common quantity is given by
$$ \frac{1}{2} \left( \prod_{i = 1}^{\infty} \frac{(1+u^i)^e}{1-qu^{2i}} + \prod_{i=1}^{\infty} \frac{1 + u^{ei}}{1-qu^{2i}} \right).$$
\end{theorem}
\begin{proof}   The first and third quantities are equal by Lemma \ref{PGUPGLReal}, while the second and fourth are equal by Lemma \ref{GLSLGUSU}.  So we just need to show the number of real classes in $\PGL(n, \FF_q)$ is equal to the number of real classes of $\GL(n, \FF_q)$ which are contained in $\SL(n, \FF_q)$.  When $q$ is even, these were already observed to be equal by Gill and Singh \cite[after Theorem 2.8]{GiSi112}, and are both equal to the number of real classes of $\GL(n, \FF_q)$ (by Lemma \ref{SLnuRevised}).  The generating function for this quantity is given by \eqref{RealClassesGen} with $e=1$, which gives our claim in this case.

We may now assume $q$ is odd, and we first calculate a generating function for the number of real classes in $\PGL(n, \FF_q)$.  If $pgl_{\nu}$ is the number of real conjugacy classes of $\PGL(n, \FF_q)$ of type $\nu$, then as stated in Lemma \ref{PGUPGLReal} we have by Gill and Singh \cite[Corollary 2.7]{GiSi112} that
$$pgl_{\nu} = \left\{ \begin{array}{ll} \displaystyle\prod_{i: m_i > 0} r_{q, m_i} & \text{ if } m_i \text{ is even for all } i, \\ \displaystyle\frac{1}{2} \prod_{i: m_i > 0} r_{q, m_i} & \text{ if } m_i \text{ is odd for some } i. \end{array} \right.$$
Applying this and the fact that
$$ \sum_{\nu \in \cP_n \atop{\exists i: m_i \text{ odd}}} \prod_{i: m_i >0} r_{q, m_i} = \sum_{\nu \in \cP_n} \prod_{i: m_i >0} r_{q, m_i} - \sum_{\nu \in \cP_n \atop{\text{all } m_i \text{ even}}} \prod_{i: m_i > 0} r_{q, m_i},$$
it follows that the generating function that we want may be written as
\begin{align}
\sum_{n \geq 0} \left( \sum_{\nu \in \cP_n} pgl_{\nu} \right) u^n & = \sum_{n \geq 0} u^n \left( \sum_{\nu \in \cP_n \atop{\text{all } m_i \text{ even}}} \prod_{i: m_i > 0} r_{q, m_i} + \frac{1}{2} \sum_{\nu \in \cP_n \atop{\exists i: m_i \text{ odd}}} \prod_{i: m_i > 0} r_{q, m_i} \right) \nonumber\\
& = \sum_{n \geq 0} u^n \left( \frac{1}{2} \sum_{\nu \in \cP_n} \prod_{i: m_i >0} r_{q, m_i} + \frac{1}{2} \sum_{\nu \in \cP_n \atop{\text{all } m_i \text{ even}}} \prod_{i: m_i > 0} r_{q, m_i} \right). \label{pglgoal}
\end{align}
From \eqref{RealClassesGen} we have
\begin{equation} \label{firsthalfpgl}
\sum_{n \geq 0} \left( \sum_{\nu \in \cP_n} \prod_{i: m_i >0} r_{q, m_i} \right) u^n = \prod_{i =1}^{\infty} \frac{(1+u^i)^2}{1-qu^{2i}}.
\end{equation}
Next, we have
\begin{align}
\sum_{n \geq 0} u^n \left( \sum_{\nu \in \cP_n \atop{\text{all } m_i \text{ even}}} \prod_{i: m_i > 0} r_{q, m_i} \right) & = \prod_{i = 1}^{\infty} \sum_{k \geq 0} (u^i)^{2k} r_{q, 2k}  = \prod_{i=1}^{\infty} \left(1 + \sum_{k \geq 1} (q^k + q^{k-1}) u^{2ik} \right) \nonumber\\
& = \prod_{i=1}^{\infty} \left( \sum_{k \geq 0} (qu^{2i})^k + \sum_{k \geq 0} u^{2i} (qu^{2i})^k \right) \nonumber\\
& = \prod_{i=1}^{\infty} \left(\frac{1}{1 - qu^{2i}} + \frac{u^{2i}}{1 - qu^{2i}} \right) = \prod_{i=1}^{\infty} \frac{1+u^{2i}}{1-qu^{2i}}. \label{halfpgl}
\end{align}
Substituting \eqref{firsthalfpgl} and \eqref{halfpgl} into \eqref{pglgoal} yields the claimed generating function.

As in Section \ref{RealClassesSLSU}, let $sl_{\nu}$ denote the number of real classes of $\GL(n, \FF_q)$ of type $\nu$ which are contained in $\SL(n, \FF_q)$.  We now compute the generating function for the number of real classes of $\GL(n, \FF_q)$ which are contained in $\SL(n, \FF_q)$, which is given by
\begin{equation} \label{slgoal}
\sum_{n \geq 0} \left( \sum_{\nu \in \cP_n} sl_{\nu} \right) u^n = \sum_{n \geq 0} u^n \left(  \sum_{\nu \in \cP_n \atop{m_i \text{ even } \forall i \text{ odd}}} sl_{\nu} + \sum_{\nu \in \cP_n \atop{\exists i \text { odd}: m_i \text{ odd}}} sl_{\nu} \right).
\end{equation} 
Consider the first sum in the parentheses of \eqref{slgoal}.  By applying Lemma \ref{SLnuRevised}, we have
\begin{align} \label{slfirsthalf} 
\sum_{n \geq 0} & u^n \left(  \sum_{\nu \in \cP_n \atop{m_i \text{ even } \forall i \text{ odd}}} sl_{\nu} \right) \nonumber\\ 
& = \frac{1}{2} \left( \sum_{n \geq 0} u^n \sum_{\nu \in \cP_n \atop{m_i \text{ even } \forall i \text{ odd}}}  \prod_{i: m_i > 0} r_{q, m_i}  + \sum_{n \geq 0} u^n \sum_{\nu \in \cP_n \atop{m_i \text{ even } \forall i \text{ odd}}} \prod_{i \text{ odd}: \atop{ m_i > 0}} \frac{q-1}{q+1} r_{q, m_i} \prod_{i \text{ even}: \atop{ m_i > 0}} r_{q, m_i} \right) \nonumber\\
& = \frac{1}{2} \left( \prod_{i = 1}^{\infty} \left(\sum_{k \geq 0} (u^{2i})^k r_{q, k} \right) \prod_{i=1}^{\infty} \left(\sum_{k \geq 0} (u^{2i-1})^{2k} r_{q, 2k} \right) \right. \nonumber\\
& \quad \quad \quad \left. + \prod_{i = 1}^{\infty} \left(\sum_{k \geq 0} (u^{2i})^{k} r_{q, k} \right) \prod_{i=1}^{\infty} \left(1 + \sum_{k \geq 1} (u^{2i-1})^{2k} \frac{q-1}{q+1} r_{q, 2k} \right) \right).
\end{align}
Note that
\begin{equation} \label{firstpart}
\prod_{i=1}^{\infty} \left( \sum_{k \geq 0} (u^{2i})^k r_{q, k} \right) = \sum_{n \geq 0} \left( \sum_{\nu \in \cP_n} \prod_{i: m_i > 0} r_{q, m_i} \right) u^{2n} = \prod_{i = 1}^{\infty} \frac{(1+u^{2i})^2}{1 - qu^{4i}},
\end{equation}
by substituting $u^2$ for $u$ in \eqref{RealClassesGen}.  Next we compute
\begin{align} \label{secondpart}
\sum_{k \geq 0} (u^{2i-1})^{2k} r_{q, 2k}  & = 1 + \sum_{k \geq 1} (q^k + q^{k-1})  (u^{2i-1})^{2k} \nonumber \\ & = \sum_{k \geq 0} (q u^{4i-2})^k + \sum_{k \geq 0} u^{4i-2} (qu^{4i-2})^k \nonumber \\
& = \frac{1+u^{4i-2}}{1 - qu^{4i-2}}.
\end{align}
Very similarly, we have
\begin{align} \label{thirdpart}
1 + \sum_{k \geq 1} \frac{q-1}{q+1} r_{q, 2k} (u^{2i-1})^{2k} & = 1 + \sum_{k \geq 1} (q^k - q^{k-1})(u^{2i-1})^{2k} \nonumber \\
& = \sum_{k \geq 0} (qu^{4i-2})^k - \sum_{k \geq 0} u^{4i-2} (qu^{4i-2})^k \nonumber \\
& = \frac{1 - u^{4i-2}}{1 - qu^{4i-2}}.
\end{align}
Substituting \eqref{firstpart}, \eqref{secondpart}, and \eqref{thirdpart} into \eqref{slfirsthalf}, we obtain
\begin{align} \label{slfirsthalf2}
\sum_{n \geq 0}  u^n \left(  \sum_{\nu \in \cP_n \atop{m_i \text{ even } \forall i \text{ odd}}} sl_{\nu} \right) & = \frac{1}{2} \left( \prod_{i = 1}^{\infty} \frac{(1 + u^{2i})^2}{1 - qu^{4i}} \frac{1 + u^{4i-2}}{1 - qu^{4i-2}} + \prod_{i = 1}^{\infty} \frac{(1 + u^{2i})^2}{1 - qu^{4i}} \frac{1 - u^{4i-2}}{1 - qu^{4i-2}} \right) \nonumber \\
& = \frac{1}{2} \left( \prod_{i = 1}^{\infty} \frac{(1 + u^{2i})^2}{1 - qu^{2i}} (1 + u^{4i-2}) + \prod_{i = 1}^{\infty} \frac{(1 + u^{2i})^2}{1 - qu^{2i}} (1 - u^{4i-2})\right)
\end{align}
Now note that
\begin{align*}
\prod_{i=1}^{\infty} (1 + u^{2i})^2 (1 - u^{4i-2}) & = \prod_{i = 1}^{\infty} \frac{(1+u^{2i})(1-u^{4i})(1-u^{4i-2})}{1-u^{2i}} \\
& = \prod_{i = 1}^{\infty} \frac{1 +u^{2i}}{1-u^{2i}} (1-u^{2i})\\
& = \prod_{i=1}^{\infty} (1 + u^{2i}).
\end{align*}
Using this, \eqref{slfirsthalf2} becomes
\begin{equation} \label{slfirsthalf3}
\sum_{n \geq 0}  u^n \left(  \sum_{\nu \in \cP_n \atop{m_i \text{ even } \forall i \text{ odd}}} sl_{\nu} \right)  = \frac{1}{2} \left( \prod_{i = 1}^{\infty} \frac{(1 + u^{2i})^2}{1 - qu^{2i}} (1 + u^{4i-2}) + \prod_{i = 1}^{\infty} \frac{1+u^{2i}}{1 - qu^{2i}} \right).
\end{equation}
For the second sum in the parentheses of \eqref{slgoal}, we again apply Lemma \ref{SLnuRevised} and compute
\begin{align} \label{slsecondhalf} 
\sum_{n \geq 0} & u^n \left(  \sum_{\nu \in \cP_n \atop{\exists i \text { odd}: m_i \text{ odd}}} sl_{\nu} \right) = \frac{1}{2}\sum_{n \geq 0}u^n \sum_{\nu \in \cP_n \atop{\exists i \text { odd}: m_i \text{ odd}}} \prod_{i: m_i > 0} r_{q, m_i} \nonumber\\
& = \frac{1}{2} \sum_{n \geq 0} u^n \left( \sum_{\nu \in \cP_n} \prod_{i: m_i > 0} r_{q, m_i} - \sum_{\nu \in \cP_n \atop{m_i \text{ even } \forall i \text{ odd}}} \prod_{i: m_i > 0} r_{q, m_i} \right) \nonumber\\
& = \frac{1}{2} \left( \prod_{i = 1}^{\infty} \left( \sum_{k \geq 0} (u^i)^k r_{q, k} \right) - \prod_{i =1}^{\infty} \left( \sum_{k \geq 0} (u^{2i})^k r_{q, k} \right) \prod_{i = 1}^{\infty} \left(\sum_{k \geq 0} (u^{2i-1})^{2k} r_{q, 2k} \right) \right).
\end{align}
Now substitute \eqref{RealClassesGen}, \eqref{firstpart}, and \eqref{secondpart} for each of the infinite products in \eqref{slsecondhalf}.  This yields
\begin{align} \label{slsecondhalf2}
\sum_{n \geq 0}  u^n \left(  \sum_{\nu \in \cP_n \atop{\exists i \text { odd}: m_i \text{ odd}}} sl_{\nu} \right) & = \frac{1}{2} \left( \prod_{i =1}^{\infty} \frac{(1 + u^i)^2}{1 - qu^{2i}} - \prod_{i = 1}^{\infty}  \frac{(1 + u^{2i})^2}{1 - qu^{4i}} \frac{1 + u^{4i-2}}{1 - qu^{4i-2}}  \right) \nonumber\\
& = \frac{1}{2} \left( \prod_{i =1}^{\infty} \frac{(1 + u^i)^2}{1 - qu^{2i}} - \prod_{i = 1}^{\infty} \frac{(1 + u^{2i})^2}{1 - qu^{2i}} (1 + u^{4i-2}) \right).
\end{align}
Now we take the sum of \eqref{slfirsthalf3} and \eqref{slsecondhalf2}, and \eqref{slgoal} becomes
\begin{equation} \label{slaccomplished}
\sum_{n \geq 0} \left( \sum_{\nu \in \cP_n} sl_{\nu} \right) u^n = \frac{1}{2} \left( \prod_{i = 1}^{\infty} \frac{(1+u^i)^2}{1-qu^{2i}} + \prod_{i=1}^{\infty} \frac{1 + u^{2i}}{1-qu^{2i}} \right),
\end{equation}
which matches our generating function for the number of real classes in $\PGL(n, \FF_q)$.
\end{proof}

\bigskip

\noindent
\begin{tabular}{ll}
\textsc{Department of Mathematics}\\ 
\textsc{College of William and Mary}\\
\textsc{P. O. Box 8795} \\
\textsc{Williamsburg, VA  23187}\\
{\em e-mail}:   {\tt ecamparo@email.wm.edu, vinroot@math.wm.edu}\\
\end{tabular}
\end{document}